\documentclass[11pt]{amsart}   

\usepackage[latin1]{inputenc}
\usepackage[T1]{fontenc}

\usepackage{amsmath}
\usepackage{amssymb}
\usepackage{amsthm}
\usepackage{mathabx}
\usepackage{MnSymbol}

\usepackage{color}
\usepackage{tikz}

\theoremstyle{plain}
\newtheorem{thm}{Theorem}[section]
\newtheorem{lem}[thm]{Lemma}
\newtheorem{prop}[thm]{Proposition}
\newtheorem{cor}[thm]{Corollary}

\theoremstyle{definition}

\theoremstyle{remark}

\newcommand{\pletH}[3]{b_{#1,\;#2}^{#3}}

\author[L. Colmenarejo]{Laura Colmenarejo} 
\thanks{The author has been partially supported by  MTM2010-19336, P12-FQM-2696, FQM-333 and FEDER}

\title[Stability properties of Plethysm]{\textsc{Stability properties of the plethysm: \\ A combinatorial approach}}
\email{laurach@us.es}

\begin{document}

\begin{abstract}
An important family of structural constants in the theory of symmetric functions and in the representation
theory of symmetric groups and general linear groups are the \emph{plethysm coefficients}. In 1950, Foulkes observed that
they have some stability properties: certain sequences of plethysm coefficients are eventually constant. Such stability
properties were proven by Brion with geometric techniques, and by Thibon and Carr\'e by means of vertex operators.

In this paper we present a new approach to prove such stability properties. Our proofs are purely combinatorial and
follow the same scheme. We decompose plethysm coefficients in terms of other plethysm coefficients related to the
complete homogeneous basis of symmetric functions. We show that these other plethysm coefficients count integer
points in polytopes and we prove stability for them by exhibiting bijections between the corresponding sets of integer points of each polytope. 
\end{abstract}

\maketitle

\section{Introduction}
\label{sec:intro}

The understanding of structural constants is one of the most important problems in representation theory. It is a difficult problem and even in simplest cases we can find unsolved problems. 

Recall that any (finite--dimensional, complex, analytic) linear representation $V$ of $GL_n (\mathbb{C})$ decomposes as a direct sum of irreducible representations:
\begin{eqnarray*}
V \approx \bigoplus m_\lambda S_\lambda(\mathbb{C}^n)
\end{eqnarray*}
where the $m_\lambda$ are non--negative integers called \emph{multiplicities} and the $S_\lambda (\mathbb{C}^n)$ are the \emph{irreducible representations of $Gl_n(\mathbb{C})$}, which are indexed by partitions of length at most $n$. 

Important families of structural constants appear in non--trivial constructions of new representations from old ones. Three of them are particularly important in a more combinatorial context. First, consider the tensor product of two irreducible representations $S_\mu(\mathbb{C}^n) \otimes S_\nu(\mathbb{C}^n)$ and decompose it into irreducible. Then, the multiplicities arising from this are the \emph{Littlewood--Richardson coefficients}, $c_{\mu \nu}^\lambda$. Combinatorial interpretations behind this family are well understood: for example, they count Littlewood-Richardson Young tableaux or integral points in the hive polytopes. Moreover, these interpretations are efficient tools for proofs and computations. 

Next, consider an irreducible representation of $GL_{mn}(\mathbb{C})$ as a representation of $GL_m(\mathbb{C})\times GL_n(\mathbb{C})$ by means of the Kronecker product of matrices, and decompose the resulting representation into irreducible. The multiplicities arising from this operation are the \emph{Kronecker coefficients}, $g_{\lambda \mu}^\nu$. In this case, combinatorial interpretations for the coefficients are known only in very particular cases. For example, C. Ballantine and R. Orellana \cite{MR2264933} have an interpretation in terms of Kronecker tableaux of the Kronecker coefficients when one of the partitions has just two parts, and M. Rosas in \cite{MR1798227} gives a description for the Kronecker product of Schur functions indexed by hook shapes and two-row shapes.

Finally, the plethysm coefficients $a_{\lambda, \mu}^\nu$ are the multiplicities obtained when we apply a Schur functor $S_\lambda$ to an irreducible representation $S_\mu(\mathbb{C}^n)$ and we decompose the resulting representation into irreducible. 

One of the major problems in combinatorial representation theory is to find interpretations for the Kronecker coefficients and the plethysm coefficients akin to those known for the Littlewood--Richardson coefficients. In this paper, we consider a family of plethysm coefficients and we give a combinatorial interpretation of them. 

Another important problem in representation theory is to understand the stability properties that we can observe for the different kind of coefficients.

For the Kronecker case, Murnaghan \cite{MR1507347} and Littlewood \cite{MR0074421} observed that some sequences stabilize (they are eventually constant). We can find results about these stability phenomena in papers of E. Vallejo \cite{MR1725703}, L. Manivel \cite{Manivel14}, I. Pak and G. Panova \cite{PakPanova14}, or E. Briand, R. Orellana and M. Rosas \cite{MR2673866}, for example.

The plethysm coefficients also show stability properties. In 1950, Foulkes, in \cite{MR0037276}, was the first one who observed some of those stability properties and in 1990's we can find the first proofs of them. 

We are interested in those phenomena. In this paper, we consider several sequences of plethysm coefficients. Thanks to our combinatorial interpretation of them, we give new combinatorial proofs of their stability properties. 

Plethysm coefficients can be computed using the language of symmetric functions: they corresponds to the coefficients
appearing in the expansion on the Schur basis $\{s_\nu\}$ of $s_{\lambda} [s_\mu]$, which is the operation, also called \emph{plethysm}, induced by the plethysm of $S_\lambda$ and $S_\mu (\mathbb{C}^n)$ in representation theory. 
So our results will be set on the symmetric function framework.
\begin{thm}\label{THM1.1}
For any partitions $\lambda$, $\mu$ and $\nu$, such that $|\lambda| \cdot |\mu| = |\nu|$, the sequences of general terms
\begin{eqnarray*}
 a_{\lambda , \mu+ (n)}^{\nu + (n \cdot |\lambda |)},
 a_{\lambda, \mu + n\pi}^{\nu + n \cdot |\lambda|\cdot \pi} ,
 a_{\lambda + (n), \mu}^{\nu + n \cdot \mu} , 
 a_{\lambda + (n), \mu}^{\nu + (n \cdot|\mu |)} 
\end{eqnarray*}
stabilize.
\end{thm}

\begin{observation}
If $|\lambda | \cdot |\mu| \neq |\nu|$, the coefficient $a_{\lambda \mu}^\nu$ is zero. So from now, we just consider partitions that satisfy this condition. 
\end{observation}

These sequences had been studied before by other authors. The following list summarizes what has been done before. 
\begin{enumerate}

 \item[$(P1)$] In \cite[Theorem 4.2]{CaTh92}, it is proved that the sequence with general term
 \begin{eqnarray*}
  a_{\lambda +(n), \mu}^{\nu + (|\mu| \cdot n)} = \left\langle s_{\lambda +(n)}[s_{\mu}], s_{\nu + (|\mu|\cdot n)} \right\rangle
 \end{eqnarray*}
 stabilizes, but with limit zero when $\ell(\mu) > 1$. 
 
 \item[$(Q1)$] In \cite[Corollary 1, Section 2.6]{Br93}, it is considered the sequence of general term
  \begin{eqnarray*}
  a_{\lambda +(n), \mu}^{\nu + n \cdot \mu} = \left\langle s_{\lambda + (n)}[s_{\mu }], s_{\nu + n \cdot \mu} \right\rangle
 \end{eqnarray*}
as a function of $n\geq 0$ and it is proved that it is an increasing function. It is also proved that it is constant for $n$ sufficiently large. 

 \item[$(R1)$] In \cite[Theorem 4.1]{CaTh92}, it is proved that the sequence with general term
 \begin{eqnarray*}
  a_{\lambda, \mu + (n)}^{\nu + (|\lambda| \cdot n)} = \left\langle s_{\lambda}[s_{\mu + (n)}], s_{\nu + (|\lambda|\cdot n)} \right\rangle
 \end{eqnarray*}
 stabilizes for $n$ big enough. 
 
 \item[$(R2)$]In \cite[Corollary 1, Section 2.6]{Br93}, it is considered the sequence of general term
  \begin{eqnarray*}
  a_{\lambda, \mu + n\cdot \pi}^{\nu + n \cdot |\lambda| \cdot \pi} = \left\langle s_{\lambda}[s_{\mu + n\cdot \pi}], s_{\nu + n \cdot |\lambda| \cdot \pi} \right\rangle
 \end{eqnarray*}
as a function of $n\geq 0$ and it is proved that it is an increasing function. It is also proved that it is constant for $n$ sufficiently large. 

\end{enumerate}

\begin{observation}
Observe that $(R1)$ is a particular case of $(R2)$, just considering $\pi=(1)$. 
\end{observation}

Before proceeding, we specify some notation we will use. We denote by $\lambda(n)$, $\mu(n)$ and $\nu(n)$ the partitions depending on $n$ that we are considering, including the case $(R2)$, in which $\mu(n)$ and $\nu(n)$ also depend on the fixed partition $\pi$.  
 The operations with partitions can be seen as operations with vectors, where we are adding zero parts in order to consider vectors of the same length. 

The stability of the sequences of plethysm coefficients $a_{\lambda(n), \mu(n)}^{\nu(n)}$ that we consider will follow from the stability of other sequences of plethysm coefficients, $b_{\lambda(n)+\omega, \mu(n)}^{\nu(n)+\omega'}$.
 These other plethysm coefficients, that we call \emph{h-plethysm coefficients}, are the coefficients of the monomial function $m_\nu$ in the plethysm $h_{\lambda}[s_\mu]$, where the $h_\lambda$ are the complete homogeneous functions. 
 We show that the coefficients $b_{\lambda, \mu}^{\nu}$ have a combinatorial interpretation in general: they count integer points in a polytope, $Q(\lambda,\mu,\nu)$. Then, in order to prove the stability for them, we build an injection from the polytope corresponding to $b_{\lambda(n), \mu(n)}^{\nu(n) }$ to
 the polytope corresponding to $b_{\lambda(n+1), \mu(n+1)}^{\nu(n+1) }$. Finally, we show that these injections are surjective for $n$ big enough. 

The idea of testing the stability of some coefficients through other coefficients can be also found in other references, as in the paper of Stembridge \cite{Stembridge}, 
where a similar idea is used for proving stability in case of \emph{Kostka coefficients} and \emph{Kronecker coefficients}. Also in the case of other plethysm coefficients, like in \cite{KhaleMichalek}, where Khale and Micha{\l}ek study stability properties for the 
plethysm coefficients $\left\langle p_\alpha[s_{(k)}], h_\beta\right\rangle$ or in the thesis of R. Abebe, \cite{Abe13}, where it is studied the plethysm coefficient $\left\langle s_\lambda[s_\mu], h_\beta\right\rangle$. 
Other techniques to prove the stability in case of plethysm coefficients can be found in papers of Manivel and Michalek \cite{MR3181730} and Brion \cite{Br93}, where they use vector bundles, or in papers of Carr\'e and Thibon \cite{CaTh92}, with vertex operators techniques. 

\section{Plethysm coefficients: a combinatorial interpretation}
\label{sec:prew}

Any (finite-dimensional, complex, analytic) linear representation of $GL_n(\mathbb{C})$ can be completely described, up to isomorphism, by its character, which is a symmetric polynomial (see \cite{MR1464693, MR1153249}). So, we can set the computations of plethysm in the framework of symmetric functions. 

Recall that the ring of symmetric functions is a graded ring endowed with a scalar product (see \cite{MR1464693, MR1354144, MR1676282}). It admits several important linear basis as the Schur functions, $s_\lambda$; the monomial functions, $m_\lambda$, or the product of complete homogeneous functions, $h_\lambda = h_{\lambda_1}h_{\lambda_2} \dots h_{\lambda_k}$. Observed that all the bases are indexed by partitions. With respect to the scalar product, the Schur functions are an orthonormal basis, and the monomial functions and the complete homogeneous functions are dual bases. 

The operation of plethysm in representation theory induces an operation $(f,g) \longrightarrow f[g]$ on the ring of symmetric functions, also called \emph{plethysm} (see \cite{MR1354144}). 
The plethysm $f[g]$ can be also seen as the evaluation of $f$ in the alphabet defined by $g[X]$ once we have written it as a sum of monomials. 
The reader can check this property for the power sum basis, for which $p_n[g]=g[p_n]$. 
This equivalent definition will be more useful for us. 
Among the properties of plethysm, we remember that this operation is associative and linear in the first argument, but it is non--commutative and not bilinear. 

In this framework, the plethysm coefficient $a_{\lambda, \mu}^\nu $ is the coefficient of $s_{\nu}$ in the expansion in the Schur basis of the plethysm of Schur functions $s_{\lambda}[s_{\mu}]$. 
Alternatively, due to the orthonormality of the Schur basis, this coefficient is extracted by means of a scalar product:
\begin{eqnarray}\label{scalarplet}
a_{\lambda \mu}^\nu =\left\langle s_{\lambda}[s_{\mu}], s_{\nu} \right\rangle.
\end{eqnarray}

The \emph{Jacobi-Trudi identity} gives us an expansion of $s_\lambda$ in terms of the complete homogeneous basis $\{h_\gamma\}$. We recall it here
\begin{lem}[Jacobi--Trudi identity, \cite{MR1354144} I. (3.4)]

 Let $\lambda$ be a partition with length at most $N$. Then
 \begin{eqnarray*}
  s_\lambda = \det \left(h_{\lambda_j + i- j} \right)_{1 \leq i, j \leq N}
 \end{eqnarray*}
 with $h_0=1$ and $h_r=0$ if $r<0$, and $\lambda$ is completed with trailing zeros if necessary.
\end{lem}
If we expand explicitly the determinant in this expression, we describe the Schur function as a sum over the permutations  $\sigma$ in the symmetric group $\mathfrak{S}_N$ (\cite{MR1354144} I. (3.4'))
\begin{eqnarray*}
 s_\lambda = \sum_{\sigma \in \mathfrak{S}_{N}} \varepsilon(\sigma) h_{\lambda + \omega(\sigma)}
\end{eqnarray*}
where $\omega(\sigma)_j  = \sigma(j)-j $, for all $j$ between $1$ and $N$, and $\varepsilon( \sigma)$ is the sign of the permutation $\sigma$.

We now perform this Jacobi--Trudi expansion for $s_{\lambda}$ and $s_{\nu}$ in \eqref{scalarplet}. We get the following alternating decomposition for the plethysm coefficients.
\begin{lem}\label{a in b}
Let $N$ and $N'$ be positive integers. Let $\lambda$, $\mu$ and $\nu$ be partitions, such that $\lambda$ has length at most $N$ and $\nu$ has length at most $N'$. Then 
\begin{eqnarray*}
a_{\lambda \mu}^\nu =
\sum_{\sigma,\tau} \varepsilon(\sigma) \varepsilon(\tau) \left\langle h_{\lambda+\omega(\sigma)}[s_{\mu}], h_{\nu+\omega(\tau)} \right\rangle
\end{eqnarray*}
where the sum is carried over all permutations $\sigma \in \mathfrak{S}_{N}$ and $\tau \in \mathfrak{S}_{N'}$.
\end{lem}

We have expressed the plethysm coefficient $a_{\lambda, \mu}^\nu$ in terms of the \emph{$h$-plethysm coefficients}, which are interesting for their own sake. For any partition $\mu$ and any finite sequences $\lambda$ and $\nu$ of integers  we set:
\begin{eqnarray*}
b_{\lambda \mu}^\nu = \left\langle h_{\lambda}[s_{\mu}], h_{\nu}\right\rangle.
\end{eqnarray*}

It turns out that these coefficients count the non-negative solutions of systems of linear Diophantine equations whose constant terms depend linearly on the parts of $\lambda$ and $\nu$. In particular, they count integer points in polytopes with an interesting description.

Before specifying a description of such polytopes, we should introduce some notation. For any partition $\mu$ and any positive integer $N$, let $t(\mu, N)$ be the set of semi--standard Young tableaux (SSYT) of shape $\mu$ with entries between $1$ and $N$,  (see \cite{MR1676282} 7.10). Let $\mathcal{P}_{\mu, N}=(\rho_j(T) )_{T,j}$ be the matrix whose rows are indexed by the tableaux $T \in t(\mu, N)$, whose columns are indexed by the integers $j$ between $1$ and $N$, such that $\rho_j(T)$ is the number of occurrences of $j$ in $T$, i.e., the row $T$ of $\mathcal{P}_{\mu, N}$, $\rho(T)$, is the \emph{weight} of the tableau $T$.

\begin{prop}\label{coeffs pleth}
Let $\lambda$ and $\nu$ be finite sequences of positive integers and let $\mu$ be a partition. Let $\ell(\mu)$ be the length of $\mu$ and let $N\geq \lambda, \nu$.

The coefficient $b_{\lambda, \mu}^\nu$ is the cardinal of the set $Q(\lambda, \mu, \nu, N)$ of matrices $\mathcal{M}=(m_{i,T})$ with non-negative integer entries whose rows are indexed by the integers $i$ between $1$ and $N$, and whose columns are indexed by the tableaux $T \in t(\mu, N)$ such that:
\begin{itemize}
\item \textsc{row sum condition of }$\mathcal{M}$: The sum of the entries in the $i-$th row of $\mathcal{M}$ is $\lambda_i$.
\item \textsc{column sum condition of} $\mathcal{MP}_{\mu, N}$: The sum of the entries in the $j-$th column of $\mathcal{MP}_{\mu, N}$ is $\nu_j$. 
\end{itemize} 
\end{prop}

\begin{observation}
We include an example of how Proposition \ref{coeffs pleth} works after its proof. 
\end{observation}

\begin{proof}

Let $x_1$, $x_2$, \ldots be the underlying variables of the symmetric functions and, for any finite sequence $\mu=(\mu_1,\mu_2,\ldots,\mu_k)$, let us denote by  $x^{\mu}=x_1^{\mu_1} x_2^{\mu_2} \cdots x_k^{\mu_k}$. 
 
The scalar product of any symmetric function with $h_{\nu}$ extracts the coefficient of $m_{\nu}$ in the expansion in basis of monomial functions. So, $\pletH{\lambda}{\mu}{\nu}$ can be interpreted 
as the coefficient of the monomial $x^{\nu}$ in $h_{\lambda}[s_{\mu}]$. 

Instead of working with symmetric functions (with infinitely many variables) we can work with symmetric polynomials in $N$ variables, provided $N$ is at least the length of $\lambda$ and $\nu$. 

We now use the expansion of the Schur polynomial $s_{\mu}(x_1,x_2,\ldots,x_N)$ into monomials (\cite{MR1464693}):
\begin{eqnarray*}
s_{\mu}(x_1,x_2,\ldots,x_N) = \sum_{T\in t(\mu,N)} x^{\rho(T)} = \sum_{T \in t(\mu, N)} x_1^{\rho_1(T)}x_2^{\rho_2(T)} \cdots x_N^{\rho_N(T)}.
\end{eqnarray*}
Since the plethysm $f[g]$ can be seen as the evaluation $f(x^{u_1}, x^{u_2}, \dots)$ once we have written $g$ as a sum of monomials, $g= \sum_i x^{u_i}$,
and the complete sum $h_n$ can be defined as the sum of all monomials of degree $n$, if $k$ is the number of SSYT of $t(\mu, N)$, we have that
\begin{eqnarray*}
h_n[s_\mu] = h_n\left[ \sum_{T \in t(\mu, N)} x_1^{\rho_1(T)}x_2^{\rho_2(T)} \cdots x_N^{\rho_N(T)} \right] = \sum_{m_1+\dots + m_k=n} \prod_{i=1}^k x^{\rho(T_i) \cdot m_i} 
\end{eqnarray*}
where $x^{ \rho(T_i) \cdot m_{i}}$ means $\left( x^{\rho(T_i)} \right)^{m_{i}} =  x_1^{\rho_1(T_i) \cdot m_{i}} \cdots x_N^{\rho_N(T_i)\cdot m_{i}}$.

We have the following decomposition in case of the complete sum $h_\lambda$
\begin{eqnarray}\label{EquationH}
h_\lambda [s_\mu] = \prod_i h_{\lambda_i} [s_\mu] = \prod_i \left(\sum_{m_{i1}+\dots + m_{ik}=\lambda_i } \prod_{j=1}^k x^{\rho(T_j)\cdot m_{ij}}\right)	  .
\end{eqnarray}

We want to extract the coefficient of $x^\nu$ in ($\ref{EquationH}$), so we take the count of the $m_{ij}$ such that:
\begin{eqnarray}\label{eq2}
 \sum_{j=1}^k m_{ij} = \lambda_i \hspace{1cm} \text{and} \hspace{1cm} \sum_{i,j} m_{ij} \cdot \rho_n(T_j)= \nu_n .
\end{eqnarray}
Finally, we just realise that the elements $m_{ij}$ form a matrix $\mathcal{M}=(m_{iT})$, whose rows are indexed by $i\in\{1,\dots,N\}$ and whose columns are indexed by the tableaux $T\in t(\mu,N)$, and the conditions in ($\ref{eq2}$) are exactly
the row sum conditions of $\mathcal{M}$ and the column sum conditions of $\mathcal{MP}_{\mu, N}$ of the Proposition. 

\end{proof}

We show an example how the result works: consider the partition $\mu=(2)$, the finite sequences $\lambda = (\lambda_1,\lambda_2)$ and $\nu=(\nu_1,\nu_2,\nu_3)$, and take $N=3$. First, we compute the set $t(\nu,N)$
\begin{center}
\begin{tikzpicture}
\draw (0,0.5) rectangle (1,1);
\draw (1.5,0.5) rectangle (2.5, 1);
\draw (3,0.5) rectangle (4,1);
\draw (4.5,0.5) rectangle (5.5,1);
\draw (6,0.5) rectangle (7,1);
\draw (7.5, 0.5) rectangle (8.5,1);

\draw (0.5,0.5) -- (0.5,1);
\draw (2,0.5) -- (2,1);
\draw (3.5,0.5) -- (3.5,1);
\draw (5,0.5) -- (5,1);
\draw (6.5,0.5) -- (6.5,1);
\draw (8,0.5) -- (8,1);

\node at (0.5,0.25) {$T_1$};
\node at (2,0.25) {$T_2$};
\node at (3.5,0.25) {$T_3$};
\node at (5,0.25) {$T_4$};
\node at (6.5,0.25) {$T_5$};
\node at (8,0.25) {$T_6$};

\node at (0.25,0.75) {$1$};
\node at (0.75,0.75) {$1$};

\node at (2.25,0.75) {$2$};
\node at (1.75,0.75) {$1$};

\node at (3.25,0.75) {$1$};
\node at (3.75,0.75) {$3$};

\node at (5.25,0.75) {$2$};
\node at (4.75,0.75) {$2$};

\node at (6.25,0.75) {$2$};
\node at (6.75,0.75) {$3$};

\node at (8.25,0.75) {$3$};
\node at (7.75,0.75) {$3$};

\end{tikzpicture}
\end{center}
Then, the corresponding $\mathcal{P}_{\mu N}$ matrix is
\begin{eqnarray*}
\left(
\begin{array}{cccccc}
2&1&1&0&0&0 \\
0&1&0&2&1&0 \\
0&0&1&0&1&2 \\
\end{array}\right)^t
\end{eqnarray*}

So, we are considering the matrices $\mathcal{M}$ with positive integer entries of the form
\begin{eqnarray*}
\left(
\begin{array}{cccccc}
m_{1T_1} & m_{1T_2} & m_{1T_3} & m_{1T_4} & m_{1T_5} & m_{1T_6} \\
m_{2T_1} & m_{2T_2} & m_{2T_3} & m_{2T_4} & m_{2T_5} & m_{2T_6} \\
m_{3T_1} & m_{3T_2} & m_{3T_3} & m_{3T_4} & m_{3T_5} & m_{3T_6} \\
\end{array}\right)
\end{eqnarray*}
that satisfy the following conditions
\begin{eqnarray*}
\begin{array}{cc}
\left\{ \begin{array}{l} \sum_{j} m_{1T_j} = \lambda_1 \\[0.3em] \sum_j m_{2T_j} = \lambda_2 \\[0.3em] \sum_j m_{3T_j} = 0 \end{array}\right. \hspace{1.6cm} & \left\{ \begin{array}{l} \sum_i \left( 2m_{iT_1} + m_{iT_2} + m_{iT_3} \right) = \nu_1 \\[0.3em] \sum_i \left( m_{iT_2} + 2m_{iT_4} + m_{iT_5} \right) = \nu_2 \\[0.3em] \sum_i \left( m_{iT_3} + m_{iT_5} + 2m_{iT_6} \right)= \nu_3 \\
 \end{array} \right. 
\end{array}
\end{eqnarray*}

The Proposition \ref{coeffs pleth} allows us to give the following result.
\begin{prop}
Fix a partition $\mu$ and a finite sequence of positive integers $\lambda$. Consider a finite sequence of integers $\nu$ such that $|\lambda| \cdot |\mu| = |\nu|$. Let $N$ be an integer bigger or equal than $\ell(\lambda)$, $\ell(\mu)$ and $\ell(\nu)$. Then, the coefficient $b_{\lambda \mu}^\nu$ will be zero unless that $\nu$ satisfies the following conditions
\begin{itemize}
\item For any $j=1,\dots, N-\ell(\mu)$, $\nu_j \leq |\lambda|\cdot \mu_1 $.
\item For any $j> N- \ell(\mu)$, $\nu_j \leq |\lambda|\cdot \mu_{j-(N-\ell(\mu))} $.
\end{itemize}

\end{prop}

\begin{observation}
Both inequalities come from the fact that we can bound the number of times that $j$ appears in any tableau $T$ of $t(\mu,N)$ and use those estimates in the $j-$th column sum condition of $\mathcal{MP}_{\mu N}$. 
\end{observation}

\section{Combinatorial Proof of Theorem \ref{THM1.1}}
\label{sec:res}

Once we have established which are the \emph{$h$-plethysm coefficients} with the combinatorial interpretation, we have to prove the stability properties for them.

We are going to separate the stability properties into two groups, depending on which partitions are depending on $n$: properties for the inner partition and properties for the outer partition,
where in the plethysm $h_\lambda [s_\mu]$, the inner partition is $\mu$ and the outer partition is $\lambda$.

\subsection{Stability properties for the outer partition}

Consider the plethysm coefficients $a_{\lambda(n), \mu(n)}^{\nu(n)}$ and $b_{\lambda(n), \mu(n)}^{\nu(n)}$ where the partition $\mu(n)$ does not depend on $n$, i.e., $\mu(n)$ is always equal to a fix partition $\mu$. 

We need a few more notation before starting with the proofs. For any sequence of integers $\alpha= (\alpha_1,\dots, \alpha_N)$, we denote by $\alpha^+$ the sequence of cumulative sums, $\alpha^+ = (\alpha_1, \alpha_1 + \alpha_2, \dots, \alpha_1 + \cdots + \alpha_{N-1})$, and by $\| \alpha \|$
the corresponding integer $\| \alpha \| = |\alpha^+ | = \sum_{j=1}^N (N+1-j)\cdot \alpha_j$.

We also define the \emph{fake length} of any sequence of non-negative integers as the length of the sequence, and we denote it by $\ell\ell(\lambda)$. For this result we need to use the fake length instead of the usual length because the zero parts that we can add to the partitions affect to the value of $\|lambda\|$ significantly. For example, consider $\lambda=(2,1,1)$. We have that $\ell\ell(\lambda)=3$ and $\|\lambda\|= 5$. But, if we write this partition adding two zeros, $\lambda_1=(2,1,1,0,0)$, we have that $\ell\ell(\lambda_1)=5 $ and $\|\lambda_1\|=13$. 

\begin{thm}\label{Q1}
 Let $\mu$ be a partition and $\lambda$ and $\nu$ be finite sequences of integers. Assume that $\ell\ell(\lambda)=\ell\ell(\mu)=\ell\ell(\nu)$. Then, the sequence with general term 
$b_{\lambda + (n), \mu}^{\nu + n\cdot\mu}$ is constant when $n \geq |\lambda | \cdot \|\mu\| - \|\nu\| - \lambda_1$.
\end{thm}

\begin{proof}
After Proposition \ref{coeffs pleth}, 
\begin{eqnarray*}
b_{\lambda + (n), \mu}^{\nu + n\cdot\mu} =\operatorname{Card} \left( Q(\lambda +(n), \mu, \nu + n \cdot\mu, N) \right).
\end{eqnarray*}
Set $E(n)$ for $Q(\lambda +(n), \mu, \nu + n\cdot \mu , N)$. 

Let $T_1$ be the tableau in $t(\mu ,N)$ whose $i-$th row is filled with occurrences of $i$, for each $i$. I.e.,
\begin{center}
 \begin{tikzpicture}
\draw (0,0) rectangle (0.5,0.5);
\node at (0.25,0.25) {1};
\draw (0.5,0) rectangle (1,0.5);
\node at (0.75,0.25) {1};
\draw  (1,0) rectangle (1.5,0.5);
\node at (1.25,0.25) {1};
\draw  (1.5,0) rectangle (2,0.5);
\node at (1.75,0.25) {1};
\draw  (2,0) rectangle (2.5,0.5);
\node at (2.25,0.25) {1};
\draw  (2.5,0) rectangle (3,0.5);
\node at (2.75,0.25) {1};

\draw (0,0.5) rectangle (0.5,1);
\node at (0.25,0.75) {2};
\draw (0.5,0.5) rectangle (1,1);
\node at (0.75,0.75) {2};
\draw  (1,0.5) rectangle (1.5,1);
\node at (1.25,0.75) {2};
\draw  (1.5,0.5) rectangle (2,1);
\node at (1.75,0.75) {2};

\draw (0,1) rectangle (0.5,1.5);
\node at (0.25,1.25) {3};
\draw (0.5,1) rectangle (1,1.5);
\node at (0.75,1.25) {3};
\draw  (1,1) rectangle (1.5,1.5);
\node at (1.25,1.25) {3};

\draw (0,1.5) rectangle (0.5,2);
\node at (0.25, 1.75) {$\dots$};
 \end{tikzpicture}.
\end{center}

So, the first row of $\mathcal{P}_{\mu,N}$ is $\rho(T_1)= \mu$. 

Consider the injection $\varphi_n : E(n) \hookrightarrow E(n+1)$ that maps any matrix $\mathcal{M} = (m_{iT}) \in E(n)$ to the matrix $\mathcal{M}^\prime=(m^\prime_{iT}) \in E(n+1)$ where the coefficient $m_{1,T_1}^\prime = m_{1,T_1} + 1$ and all other coefficients are unchanged.

We contend that $\varphi_n$ is also surjective for $n$ big enough. The map $\varphi_n$ is surjective if and only if for all $\mathcal{M}^\prime = (m^\prime_{iT}) \in E(n+1)$, the entry $m_{1,T_1}^\prime$ is non--zero. 
So, in order to prove surjectivity, we will show that $m_{1,T_1}^\prime >0$.

Let $\mathcal{M}^\prime = (m^\prime_{iT})  \in E(n+1)$. Observe that among all tableaux in $t(\mu ,N)$, the tableau $T_1$ is the unique one with maximum weight for the \emph{dominance ordering} (\cite{MR1354144} I.1). 
Then,
\begin{eqnarray}\label{est1}
\left\{ \begin{array}{lcl}
|| \rho(T) || \leq ||\mu ||-1 & \text{ if } & T \neq T_1 \\[0.3em]
|| \rho (T_1) || = ||\mu || & \text{ if } & T = T_1
\end{array} \right.
\end{eqnarray}

Looking at the column sum conditions for $\mathcal{M}^\prime \mathcal{P}_{\mu ,N}$, for any $j$, we know that 
\begin{eqnarray*}
\sum_{i, T} m_{iT} \cdot \rho_j(T) = \nu_j + (n+1)\cdot \mu_j .
\end{eqnarray*} 

Consider the cumulative sums
\begin{eqnarray*}
 \sum_j \sum_{i,T} m_{iT}  \cdot \rho_j^+(T)= \sum_j \left(\nu_j^+ + (n+1)  \cdot \mu_j^+\right) .
 \end{eqnarray*}
We can write this as
\begin{eqnarray*}
\sum_{i,T} m_{iT}  \cdot ||\rho(T)||= ||\nu || + (n+1) \cdot  ||\mu ||  .
\end{eqnarray*}
Isolating the factor corresponding to $T_1$ from the other tableaux and applying the estimates in ($\ref{est1}$) to each part, we obtain that
\begin{eqnarray*}
||\nu ||+ (n+1)\cdot ||\mu || \leq  ||\mu ||\cdot \sum_i m_{iT_1} + \left(||\mu ||-1\right) \cdot \sum_{i, T\neq T_1} m_{iT} .
\end{eqnarray*}
Now, extracting the factor $(\|\mu \| -1)$ 
\begin{eqnarray*}
||\nu ||+ (n+1)\cdot ||\mu || &\leq& \left(||\mu ||-1\right) \cdot \left( \sum_i m_{iT_1} + \sum_{i, T\neq T_1} m_{iT} \right) + \sum_i m_{iT_1} .
\end{eqnarray*}
Observe that if $i>1$, using the row sum conditions for $\mathcal{M}$, $m_{iT_1} \leq \lambda_i$. So the sum of all entries different from $m_{1T_1}$ in the first column of $\mathcal{M}$ is, at most, $|\overline{\lambda}|$. 
Recalling that $\sum_{i,T} m_{iT} = |\lambda | + (n+1)$, we obtain that
\begin{eqnarray*}
||\nu ||+ (n+1) \cdot||\mu || &\leq& \left(||\mu ||-1\right) \cdot \left( |\lambda | + (n+1) \right) + m_{1T_1} + |\overline{\lambda}| .
\end{eqnarray*}
This inequality is simplified as
\begin{eqnarray}\label{eq54}
 \|\nu \| + (n+1) -  \| \mu \| \cdot |\lambda | + \lambda_1 \leq m_{1T_1} .
\end{eqnarray}

Looking at ($\ref{eq54}$), $m_{1T_1}> 0$ as soon as $\|\nu \| + (n+1) -  \| \mu \| \cdot |\lambda |  + \lambda_1> 0$. 

Then, for $n \geq \| \mu\| \cdot |\lambda | - \|\nu\| - \lambda_1$, we have proved that our sequence is constant. 

\end{proof}

\begin{observation}
The relation between the coefficients $b_{\lambda \mu}^\nu$ and $a_{\lambda, \mu}^\nu$ in $\eqref{a in b}$ shows that the stability in one case implies the stability in the other one. 
So, as a consequence we have that the sequence with general term $a_{\lambda + (n), \mu}^{\nu + n\cdot\mu}$ stabilizes. 
\end{observation}

We now use again the main idea of the proof of Theorem \ref{Q1} for another stability property. 
\begin{thm}\label{P1}
Let $\mu$ be a partition and $\lambda$ and $\nu $ be finite sequences of integers. Consider the following \emph{$h$-plethysm coefficients}
\begin{eqnarray*}
b_{\lambda + (n), \mu}^{\nu +\left(|\mu |\cdot n\right)}= \left\langle h_{\lambda + (n)} [s_\mu], h_{\nu + \left(|\mu |\cdot n\right)} \right\rangle
\end{eqnarray*}
The sequence of coefficients stabilizes. It has limit zero whenever $\ell(\mu) > 1$.
\end{thm}

\begin{observation}
 The case in which $\mu$ has one part is a particular case of Theorem \ref{Q1}. We include it in the proof because we obtain a better bound.  
\end{observation}

\begin{proof}
Let $N$ be an integer bigger than or equal to the lengths of $\lambda$ and $\nu$. By Proposition \ref{coeffs pleth},
\begin{eqnarray*}
b_{\lambda + (n), \mu}^{\nu +\left(|\mu | \cdot n\right)}= \text{Card } \left( Q(\lambda + (n), \mu, \nu +\left(|\mu |\cdot n\right), N) \right).
\end{eqnarray*}
Set $E(n) = Q(\lambda + (n), \mu, \nu +\left(|\mu |\cdot n\right), N)$.

If $\ell(\mu) > 1$, we need to prove that $E(n)$ is empty for $n$ big enough. \newline
The elements in $E(n)$ are matrices satisfying the row sum conditions and the column sum conditions. 
These conditions can be written as a system of equations. For example, the first column sum condition of $\mathcal{MP}_{\mu, N}$ says that
\begin{eqnarray}\label{FirstColumnCond}
\nu_1 +|\mu | \cdot n= \sum_{i,T} m_{i,T}\cdot \rho_{1}(T).
\end{eqnarray}
Then, $E(n)$ is the set of solutions of this system. 

We want to show that this system has no solution, when $n$ is sufficiently large. 

We observe that each tableau has, at most, as many boxeslabeled with one as boxes are in the first row:
\begin{eqnarray}\label{Estimation1}
 \rho_1(T) \leq \mu_1 .
\end{eqnarray}

So, using ($\ref{Estimation1}$) in the equation ($\ref{FirstColumnCond}$),
\begin{eqnarray*}
 \nu_1 +|\mu |\cdot n \leq \mu_1 \cdot \left(\sum_{i,T} m_{iT}\right) = \mu_1 \cdot \left( |\lambda | + n \right)  .
\end{eqnarray*}

We can conclude that when $n > \frac{\mu_1\cdot |\lambda | - \nu_1}{|\mu | - \mu_1}$, there is no solution for the system, and $E(n)$ is an empty set. \\
Then, $b_{\lambda + (n), \mu}^{\nu + (|\mu | \cdot n)} =0 $ when $n > \frac{\mu_1\cdot |\lambda | - \nu_1}{|\mu | - \mu_1}$.

If $\ell(\mu)=1$, $\mu$ has the following form: $\mu=(m)$, for some positive integer $m$.
We want to prove that there exists a bijection between $E(n)$ and $E(n+1)$, when $n$ is big enough. \\
Let $T_1$ be the tableau in $t(\mu, N)$ which is filled just with ones:
\begin{center}
 \begin{tikzpicture}
\draw (0,0) rectangle (0.5,0.5);
\node at (0.25,0.25) {1};
\draw (0.5,0) rectangle (1,0.5);
\node at (0.75,0.25) {1};
\draw  (1,0) rectangle (1.5,0.5);
\node at (1.25,0.25) {$\dots$};
\draw  (1.5,0) rectangle (2,0.5);
\node at (1.75,0.25) {1};
\draw  (2,0) rectangle (2.5,0.5);
\node at (2.25,0.25) {1};
\draw  (2.5,0) rectangle (3,0.5);
\node at (2.75,0.25) {1};
\end{tikzpicture}
\end{center}

We define the map as in Theorem \ref{Q1} 
\begin{eqnarray*}
\begin{array}{rccc}
\varphi_n : & E(n) & \longrightarrow & E(n+1) \\
 & \mathcal{M} = (m_{iT}) & \longmapsto & \mathcal{M}^\prime = (m_{iT}^\prime)
\end{array}
\end{eqnarray*}
where  $m_{1,T_1}^\prime = m_{1,T_1} + 1$ and all other coefficients of $\mathcal{M}^\prime$ are unchanged. 

Since $\varphi_n$ is well-defined and injective, we just need to prove that it is also surjective, i.e., we have to show that for all $\mathcal{M}^\prime \in E(n+1)$, $m_{1,T_1}^\prime > 0$. 

In this case, we do not need to consider the cumulative sums. 
It is enough to observe that 
\begin{eqnarray}\label{Estimation2}
\left\{ \begin{array}{lcl}
\rho_1(T_1) =m  &  \\[0.3em]
\rho_1 (T) \leq m-1 & \text{ if } & T \neq T_1.
\end{array} \right.
\end{eqnarray}

Consider the first column sum condition for $\mathcal{M}^\prime\mathcal{P}_{\mu, N}$
\begin{eqnarray*}
\nu_1 + (n+1)\cdot m &=& \sum_{i,T} m_{iT} \cdot \rho_1(T) .
\end{eqnarray*}
Isolating the case $T=T_1$ from the others and applying the estimates for $\rho_1(T)$ in ($\ref{Estimation2}$), we obtain that
\begin{eqnarray*}
\nu_1 + (n+1)\cdot m &\leq& (m-1) \cdot \sum_{i, T\neq T_1} m_{iT} + m \cdot \sum_i m_{iT_1} 
\end{eqnarray*}
Observe that if $i>1$, using the row sum conditions for $\mathcal{M}$, $m_{iT_1} \leq \lambda_i$. So the sum of all entries different from $m_{1T_1}$ in the first column of $\mathcal{M}$ is, at most, $|\overline{\lambda}|$. 
Recalling that $\sum_{i,T} m_{iT} = |\lambda | + (n+1)$, we obtain the following inequality
\begin{eqnarray*}
\nu_1 + (n+1)\cdot m &\leq& (m-1)\cdot (|\lambda | + (n+1)) + m_{1T_1} + |\overline{\lambda} |
\end{eqnarray*}
that is simplified as
\begin{eqnarray*}
\nu_1 &\leq& m \cdot |\lambda | - (n+1) + m_{1T_1} - \lambda_1 .
\end{eqnarray*}
This means that $m_{1,T_1} > 0$ as soon as $n\geq |\overline{\nu}|- \lambda_1 -1$. 

\end{proof}

\begin{observation}
It follows from Theorem \ref{P1} that the sequence 
\begin{eqnarray*}
a_{\lambda + (n), \mu}^{\nu +\left(|\mu |\cdot n\right)}= \left\langle s_{\lambda + (n)} [s_\mu], s_{\nu + \left(|\mu |\cdot n\right)} \right\rangle
\end{eqnarray*}
stabilizes and it has limit zero when $\ell(\mu) >1$.
\end{observation}

\subsection{Stability properties for the inner partition}
In this case, the set of tableaux of shape $\mu(n)$ changes when $n$ grows. 
\begin{thm}\label{R1}
Let $\mu$ be a partition and $\lambda$ and $\nu$ be finite sequences of integers. The following \emph{$h$-plethysm coefficients}
\begin{eqnarray*}
b_{\lambda , \mu + (n)}^{\nu +\left(|\lambda |\cdot n\right)}= \left\langle h_{\lambda} [s_{\mu + (n)}], h_{\nu + \left(|\lambda |\cdot n\right)} \right\rangle
\end{eqnarray*}
stabilize for $n$ big enough.
\end{thm}

\begin{observation}
For each $T$, let $M_T$ be the column of $\mathcal{M}$ associated to the tableau $T$ and denote by $|M_T|$ the sum of the entries of $M_T$. Then, the condition $\sum_{i,T} m_{iT} = |\lambda |$ becomes $\sum_{T} |M_T | = |\lambda |$ and each column sum condition for $\mathcal{MP}_{\mu ,N}$ is written as $\sum_i |M_T| \cdot \rho_j(T) = \nu_j$.
\end{observation}

\begin{proof}
Let $N$ be an integer bigger than or equal to the lengths of $\lambda$ and $\nu$. By Proposition \ref{coeffs pleth},
\begin{eqnarray*}
b_{\lambda , \mu + (n)}^{\nu +\left(|\lambda |\cdot n\right)}= \text{Card } \left( Q(\lambda , \mu + (n), \nu +\left(|\lambda |\cdot n\right), N) \right).
\end{eqnarray*}
Set $E(n) = Q(\lambda , \mu + (n), \nu +\left(|\lambda |\cdot n\right), N)$.

Since the set of tableaux changes in each step, first of all we define the following map
\begin{eqnarray*}
\begin{array}{rccc}
\tau_n : & t(\mu + (n), N) & \longrightarrow & t(\mu + (n+1), N) 
\end{array}
\end{eqnarray*}
where $\tau_n(T)$ is obtained from $T$ by adding one box labelled by one in the first row and pushing the original first row of $T$ to the right. For instance,
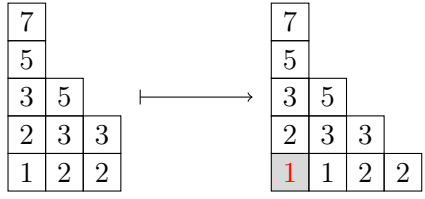
\begin{center}
 \begin{tikzpicture}
\draw (0,0) rectangle (0.5,0.5);
\node at (0.25,0.25) {1};
\draw (0.5,0) rectangle (1,0.5);
\node at (0.75,0.25) {2};
\draw (1,0) rectangle (1.5,0.5);
\node at (1.25,0.25) {2};
\draw (0,0.5) rectangle (0.5,1);
\node at (0.25,0.75) {2};
\draw (0.5,0.5) rectangle (1,1);
\node at (0.75,0.75) {3};
\draw (1,0.5) rectangle (1.5,1);
\node at (1.25,0.75) {3};
\draw (0,1) rectangle (0.5,1.5);
\node at (0.25,1.25) {3};
\draw (0.5,1) rectangle (1,1.5);
\node at (0.75,1.25) {5};
\draw (0,1.5) rectangle (0.5,2);
\node at (0.25,1.75) {5};
\draw (0,2) rectangle (0.5,2.5);
\node at (0.25,2.25) {7};

\draw[|->] (1.75,1.25) -- (3.25,1.25);

\draw[fill=gray!30] (3.5,0) rectangle (4,0.5);
\node at (3.75,0.25) {\color{red} {1}};
\draw (4,0) rectangle (4.5,0.5);
\node at (4.25,0.25) {1};
\draw (4.5,0) rectangle (5,0.5);
\node at (4.75,0.25) {2};
\draw (5,0) rectangle (5.5,0.5);
\node at (5.25,0.25) {2};
\draw (3.5,0.5) rectangle (4,1);
\node at (3.75,0.75) {2};
\draw (4,0.5) rectangle (4.5,1);
\node at (4.25,0.75) {3};
\draw (4.5,0.5) rectangle (5,1);
\node at (4.75,0.75) {3};
\draw (3.5,1) rectangle (4,1.5);
\node at (3.75,1.25) {3};
\draw (4,1) rectangle (4.5,1.5);
\node at (4.25,1.25) {5};
\draw (3.5,1.5) rectangle (4,2);
\node at (3.75,1.75) {5};
\draw (3.5,2) rectangle (4,2.5);
\node at (3.75,2.25) {7};
 \end{tikzpicture}
\end{center}

The map $\tau_n$ is injective but not surjective. This allows us to separate the set $t(\mu + (n+1), N)$ into two sets: 
the set of \emph{tableaux with pre-image}, whose tableaux are denoted by $\overline{T}$, and the set of \emph{new tableaux}, whose tableaux are denoted by $T^\prime$. 

In order to define $\varphi_n : E(n) \longrightarrow E(n+1)$, we order the columns of the matrices. The first columns correspond to the tableaux with pre-image and the next columns correspond to the new tableaux. 
So, the matrices have the following form: $\mathcal{M} = \left( M_{\overline{T}} \left| \right. M_{T^\prime} \right) $.\\
Then, we define $\varphi_n (\mathcal{M})= \left(\mathcal{M} \left| \right. \overline{0} \right)$. 
This means that we turn the matrix of $E(n)$, $\mathcal{M}$, into a matrix of $E(n+1)$ by adding as many null columns as we need. In fact, we add in total the same number of columns as new tableaux there are in $t(\mu + (n+1), N)$. 

The map $\varphi_n$ is well defined and injective. We check surjectivity.

Let $\mathcal{M}^\prime \in E(n+1)$. Let denote $S=\sum_{\overline{T}} |M^\prime_{\overline{T}}|$ and $S^\prime = \sum_{T^\prime} |M^\prime_{T^\prime}|$. 
Then, $S+S^\prime = |\lambda |$ and we need to show that $S^\prime =0$. From this it will follow that, for all new tableau $T^\prime$, $M_{T^\prime}^\prime =\overline{0}$ and that $\mathcal{M}^\prime$ is of the form $(\mathcal{M} \left| \right. \overline{0})$, with $\mathcal{M} \in E(n)$. 

We can restate the first column sum condition for $\mathcal{M}^\prime \mathcal{P}_{\mu,N}$ as the following condition
\begin{eqnarray}\label{eq1}
\nu_1 + (n+1) \cdot |\lambda | = \sum_{\overline{T}} |M_{\overline{T}}| \cdot \rho_1(\overline{T}) + \sum_{T^\prime} |M_{T^\prime}| \cdot \rho_1(T^\prime)
\end{eqnarray}
We estimate the number of ones in each case
\begin{eqnarray}\label{Estimation3}
\left\{ \begin{array}{lcl}
\rho_1(T) \leq \mu_1+n+1 & \text{ if } & T= \overline{T} \\[0.3em]
\rho_1(T) \leq \mu_2 & \text{ if } & T = T^\prime .
\end{array} \right.
\end{eqnarray}
For tableaux with pre-image, we can have as many ones as boxes are in the first row. For new tableaux, we cannot have more ones than the length of the second row; otherwise, there would exist the pre-image of the tableau. 

Using these estimates ($\ref{Estimation3}$) in ($\ref{eq1}$), we get that
\begin{eqnarray*}
\nu_1 + (n+1)\cdot |\lambda | &\leq&  (\mu_1 + n )\cdot S + \mu_2 \cdot S^\prime .
\end{eqnarray*}
Recalling that $S^\prime = |\lambda | -S$ and reorganizing the inequality, we obtain that
\begin{eqnarray*} 
S^\prime \leq \frac{\mu_1 \cdot |\lambda | - \nu_1}{\mu_1 + n +1-\mu_2} 
\end{eqnarray*}
Since $S^\prime \geq 0$, as soon as $n\geq  \mu_1 \cdot \left(|\lambda |-1\right) + \mu_2 - \nu_1$, $S^\prime$ will be zero.

\end{proof}

\begin{observation} From Theorem \ref{R1}, we obtain the corresponding property for the $a_{\lambda(n), \mu(n)}^{\nu(n)}$ coefficients using $\eqref{a in b}$.
\end{observation}

Once we have proved $(R1)$, we proceed with $(R2)$. 
\begin{thm}\label{R2}
Let $\mu$ and $\pi$ be partitions. Let $\lambda$ and $\nu$ be positive integer sequences and $n\in \mathbb{N}$. Then, the following sequence stabilizes 
\begin{eqnarray*}
b_{\lambda, \mu+n \cdot\pi}^{\nu + n\cdot |\lambda | \cdot \pi} = \left\langle h_\lambda[s_{\mu+n \cdot\pi}], h_{\nu + n\cdot |\lambda | \cdot \pi}\right\rangle
\end{eqnarray*}
when $n$ is big enough.
\end{thm}

\begin{proof}
Let $N$ be an integer bigger than or equal to the lengths of $\lambda$ and $\nu$. By Proposition \ref{coeffs pleth}, 
\begin{eqnarray*}
 b_{\lambda, \mu+n \cdot\pi}^{\nu + n\cdot |\lambda | \cdot \pi} = \text{Card }\left( Q(\lambda, \mu + n \cdot\pi, \nu +n\cdot |\lambda | \cdot \pi, N)\right) .
\end{eqnarray*}
Set $E(n)= Q(\lambda, \mu + n \cdot\pi, \nu +n\cdot |\lambda | \cdot \pi, N)$.

Before defining a map between $E(n)$ and $E(n+1)$, we define a map between the set of tableaux
\begin{eqnarray*}
\begin{array}{rccc}
\tau_n : &t(\mu +  n\cdot\pi, N) & \longrightarrow & t(\mu + (n+1)\cdot\pi, N) \\
 & T & \longrightarrow & \tau_n(T)
\end{array}
\end{eqnarray*}
where $\tau_n(T)$ is obtained from $T$ adding in the left side the SSYT of shape $\pi$, which has $\pi_i$ boxes filled with $i$'s in the $i-$th row, and pushing the original rows of $T$ to the right. 
For instance,
\begin{center}
 \begin{tikzpicture}
\draw (0,0) rectangle (0.5,0.5);
\node at (0.25,0.25) {1};
\draw (0.5,0) rectangle (1,0.5);
\node at (0.75,0.25) {2};
\draw (1,0) rectangle (1.5,0.5);
\node at (1.25,0.25) {2};
\draw (0,0.5) rectangle (0.5,1);
\node at (0.25,0.75) {2};
\draw (0.5,0.5) rectangle (1,1);
\node at (0.75,0.75) {3};
\draw (1,0.5) rectangle (1.5,1);
\node at (1.25,0.75) {3};
\draw (0,1) rectangle (0.5,1.5);
\node at (0.25,1.25) {3};
\draw (0.5,1) rectangle (1,1.5);
\node at (0.75,1.25) {5};
\draw (0,1.5) rectangle (0.5,2);
\node at (0.25,1.75) {5};
\draw (0,2) rectangle (0.5,2.5);
\node at (0.25,2.25) {7};

\draw[|->] (1.75,1.25) -- (3.25,1.25);

\draw[fill=gray!30] (3.5,0) rectangle (4,0.5);
\node at (3.75,0.25) {\color{red} {1}};
\draw[fill=gray!30] (4,0) rectangle (4.5,0.5);
\node at (4.25,0.25) {\color{red} {1}};
\draw[fill=gray!30] (4.5,0) rectangle (5,0.5);
\node at (4.75,0.25) {\color{red} {1}};
\draw (5,0) rectangle (5.5,0.5);
\node at (5.25,0.25) {1};
\draw (5.5,0) rectangle (6,0.5);
\node at (5.75,0.25) {2};
\draw (6,0) rectangle (6.5,0.5);
\node at (6.25,0.25) {2};
\draw[fill=gray!30] (3.5,0.5) rectangle (4,1);
\node at (3.75,0.75) {\color{red} {2}};
\draw[fill=gray!30] (4,0.5) rectangle (4.5,1);
\node at (4.25,0.75) {\color{red} {2}};
\draw (4.5,0.5) rectangle (5,1);
\node at (4.75,0.75) {2};
\draw (5,0.5) rectangle (5.5,1);
\node at (5.25,0.75) {3};
\draw (5.5,0.5) rectangle (6,1);
\node at (5.75, 0.75) {3};
\draw[fill=gray!30] (3.5,1) rectangle (4,1.5);
\node at (3.75,1.25) {\color{red} {3}};
\draw[fill=gray!30] (4,1) rectangle (4.5,1.5);
\node at (4.25,1.25) {\color{red} {3}};
\draw (4.5,1) rectangle (5,1.5);
\node at (4.75, 1.25) {3};
\draw (5,1) rectangle (5.5,1.5);
\node at (5.25,1.25) {5};
\draw (3.5,1.5) rectangle (4,2);
\node at (3.75,1.75) {5};
\draw (3.5,2) rectangle (4,2.5);
\node at (3.75,2.25) {7};
 \end{tikzpicture}
\end{center}

The map $\tau_n$ is injective but not surjective. So, we separate $t(\mu + (n+1)\cdot \pi, N)$ into two sets: the set of \textit{tableaux with pre-image}, 
whose elements we denote by $\overline{T}$, and the set of \textit{new tableaux}, whose elements we denoted by $T^\prime$.

We define a map between $E(n)$ and $E(n+1)$. First, we order the columns of the matrices: the first columns correspond to the tableaux with pre-image, $\overline{T}$,
and the next columns correspond to the new tableaux, $T^\prime$. So, the matrices have the following form: $\mathcal{M} = \left( M_{\overline{T}} \left| M_{T^\prime} \right. \right)$. Then,
\begin{eqnarray*}
\begin{array}{rccc}
\varphi_n: & E(n) & \longrightarrow & E(n+1) \\
 & \mathcal{M} & \longrightarrow & \mathcal{M}^\prime = \left( \mathcal{M} | \overline{0}\right)
\end{array}
\end{eqnarray*}
where we are adding to $\mathcal{M}$ a null column for each new tableaux in $t(\mu + (n+1)\cdot \pi ,N)$. 

The map $\varphi_n$ is well defined and  injective. We want to show that it is also surjective. 

In this case it is not as easy as in Theorem \ref{R1} to obtain estimates for the weights of the tableaux. We need to introduce another combinatorial object, the \textit{Gelfand-Tsetlin patterns}.
They are triangular arrays, $G$, of non-negative integers, say
\begin{eqnarray*}
 \begin{array}{ccccccccc}
  x_{n1} & & x_{n2} &  & \cdots &  & x_{n,n-1} & & x_{nn} \\
   & \ddots & & \ddots & & \ddots & & \udots & \\
   & & x_{31} & & x_{32} & & x_{33} & & \\
   & & & x_{21} & & x_{22} & & & \\
   & & & & x_{11} & & & & \\
 \end{array}
\end{eqnarray*}
such that $x_{i+1,j+1} \leq x_{i,j} \leq x_{i+1,j}$, for all $1\leq j \leq i \leq n$, when all three numbers are defined.
 
There exists a bijection between SSYT and Gelfand-Tsetlin patterns. Let $T$ be a SSYT of shape $\mu$ and weight $\beta$, the bijection is defined in the following way: $x_{ij}$ is the number of entries in the $j-$th row of $T$ that are $\leq i$.
So, the elements of the last row will be $x_{ki}= \mu_i$ and $\sum_{j} x_{ij} = \beta_1 + \dots + \beta_i$. 
For instance, 
\begin{center}
 \begin{tikzpicture}
\draw (0,0) rectangle (0.5,0.5);
\node at (0.25,0.25) {1};
\draw (0.5,0) rectangle (1,0.5);
\node at (0.75,0.25) {2};
\draw (1,0) rectangle (1.5,0.5);
\node at (1.25,0.25) {2};
\draw (1.5,0) rectangle (2,0.5);
\node at (1.75,0.25) {3};
\draw (0,0.5) rectangle (0.5,1);
\node at (0.25,0.75) {2};
\draw (0.5,0.5) rectangle (1,1);
\node at (0.75,0.75) {3};
\draw (1,0.5) rectangle (1.5,1);
\node at (1.25,0.75) {3};
\draw (1.5,0.5) rectangle (2,1);
\node at (1.75, 0.75) {4};
\draw (0,1) rectangle (0.5,1.5);
\node at (0.25,1.25) {3};
\draw (0.5,1) rectangle (1,1.5);
\node at (0.75,1.25) {5};
\draw (0,1.5) rectangle (0.5,2);
\node at (0.25,1.75) {5};
\draw (0.5,1.5) rectangle (1,2);
\node at (0.75,1.75) {6};

\draw[|->] (2.25,1) -- (3.75,1);

\node at (4,2.25) {4};
\node at (5,2.25) {4};
\node at (6,2.25) {2};
\node at (7,2.25) {2};
\node at (8,2.25) {0};
\node at (9,2.25) {0};

\node at (4.5,1.75) {4};
\node at (5.5, 1.75) {4};
\node at (6.5,1.75) {2};
\node at (7.5,1.75) {1};
\node at (8.5,1.75) {0};

\node at (5,1.25) {4};
\node at (6,1.25) {4};
\node at (7,1.25) {1};
\node at (8,1.25) {0};

\node at (5.5,0.75) {4};
\node at (6.5,0.75) {3};
\node at (7.5,0.75) {1};

\node at (6,0.25) {3};
\node at (7,0.25) {1};

\node at (6.5, -0.25) {1};
\end{tikzpicture}
\end{center}

These two conditions characterize the Gelfand-Tsetlin pattern. 
We denote by $GF(T)$ the Gelfand-Tsetlin pattern associated to $T$ and by $GF(\mu)$ the set of Gelfand-Tsetlin pattern associated to SSYT of shape $\mu$.

The map defined between the sets of tableaux, $\tau_n$, induces a map between the Gelfand-Tsetlin patterns in the following way:  
\begin{eqnarray*}
 \begin{array}{rccc}
  \tau^\prime_n: & GF(\mu + n\cdot \pi) & \longrightarrow & GF(\mu + (n+1)\cdot \pi) \\
  & GF(T)=(x_{ij}) & \longrightarrow & GF(\tau_n(T))= (x_{ij}+ \pi_j) = (y_{ij})
 \end{array}
\end{eqnarray*}

Before proving the surjectivity of $\varphi_n$, for $n$ big enough, we are going to prove the following lemma.
\begin{lem}\label{AuxLem}
For each $i$, there exists a constant $c_i$ such that if $T$ is a tableau of $ t(\mu + n\cdot \pi , N)$ for which the number of $i$'s in the $i-$th row is less than or equal to $\mu_i + n\cdot \pi_i - c_i$, for some $i$,  then $M_T= \overline{0}$.  

In fact, $c_i > \sum_{j=1}^i \left( |\lambda | \cdot \mu_j - \nu_j\right)  $
\end{lem}

\begin{proof}
Suppose that there exists $T_0$ such that for all $i$, the number of $i$'s in the $i-$th row is less than or equal to $\mu_i + n\pi_i - c_i$ and $M_{T_0} \neq \overline{0}$. 
We want to get a contradiction at some point. 

In general, for any SSYT of $t(\mu + n\cdot \pi, N)$, $T$, we have the following estimate 
\begin{eqnarray}\label{EqEst1}
 \rho_1(T) + \cdots + \rho_i(T) \leq (\mu_1 + n\cdot \pi_1) + \cdots + (\mu_i + n\cdot \pi_i)
\end{eqnarray}
but in case of $T_0$, we can refine this bound
\begin{eqnarray}\label{EqEst2}
 \rho_1(T_0) + \cdots + \rho_i(T_0) \leq (\mu_1 + n\cdot \pi_1) + \cdots + (\mu_i + n\cdot \pi_i) - c_i
\end{eqnarray}

The sum of the first $i$ column sum conditions for $\mathcal{MP}_{\mu + n\cdot \pi,N}$ can be restated as
\begin{eqnarray*}
\sum_{j=1}^i \left(\nu_j + n\cdot |\lambda |\cdot \pi_j\right)  = 
 \sum_{T\neq T_0} |M_T| \cdot \sum_{j=1}^i \rho_j(T) + |M_{T_0}| \cdot \sum_{j=1}^i \rho_j(T_0) .
 \end{eqnarray*}
And using the estimates ($\ref{EqEst1}$) and ($\ref{EqEst2}$), we obtain that 
\begin{multline*}
\sum_{j=1}^i \left(\nu_j + n\cdot |\lambda |\cdot \pi_j\right)  \leq \sum_{j=1}^i \left(\mu_j + n\cdot \pi_j\right)\cdot \left(|\lambda |- |M_{T_0}|\right) + \\ 
+ \sum_{j=1}^i \left(\mu_j + n\cdot \pi_j\right)\cdot |M_{T_0}| - |M_{T_0}| \cdot c_i .
\end{multline*}
Reorganizing last inequality, we get that
\begin{eqnarray*}
 |M_{T_0}| \cdot c_i \leq \sum_{j=1}^i \left(|\lambda | \cdot \mu_j - \nu_j\right) .
\end{eqnarray*}

So, it is enough to consider $c_i > \sum_{j=1}^i \left( |\lambda | \cdot \mu_j - \nu_j\right) $ in order to obtain a contradiction, as we wanted. 

\end{proof}

Finally, we are ready for checking the surjectivity of $\varphi_n$.

Let $\mathcal{M}^\prime\in E(n+1)$. We need to show that all the columns of $\mathcal{M}^\prime$ corresponding to new tableaux are null columns. Using Lemma \ref{AuxLem},
we need to show that the set of tableaux of shape $\mu + (n+1)\cdot \pi$ such that, for all $i$, the number of $i$'s in the $i-$th row is bigger than $\mu_i + (n+1)\cdot \pi_i - c_i$, 
for the constants of the Lemma, is a subset of the set of tableaux of shape $\mu + (n+1)\cdot \pi$ with pre-image. 

Let $T$ be a tableau of shape  $\mu + (n+1)\cdot \pi$  such that the number of $i$'s in the $i-$th row is bigger than $\mu_i + (n+1)\cdot \pi_i - c_i$ for all $i$. We denote this condition by ($\star$), for further references. 

This tableau $T$ has associated a Gelfand-Tsetlin pattern, $GF(T)=(y_{ij})$. We define the pre-image of $GF(T)$, $(\tau^\prime)^{-1} (GF(T))=(x_{ij})$ in $GF(\mu+ (n+1)\cdot \pi )$, by setting $x_{ij}=y_{ij} - \pi_j$. 
Then, we build the pre-image of $T$, $\varphi_n^{-1}(T)$, by considering the associated tableau to $(\tau^\prime)^{-1} (GF(T))$. 

We need to prove that this pre-image of $GF(T)$ is well-defined. There are three kinds of inequalities we should check:
\begin{itemize}
 \item To show that $x_{i+1,j} \geq x_{ij}$, we use directly the definition of the pre-image and that $(y_{ij})$ is a Gelfand-Tsetlin pattern. 

 \item We show that $x_{ij} \geq x_{i+1,j+1}$ \\
 If $\pi_j= \pi_{j+1}$, we have that
  \begin{eqnarray*}
  x_{ij} \geq   x_{i+1,j+1} 
 \end{eqnarray*}
directly from the fact that $(y_{ij})$ is a Gelfand-Tsetlin pattern. 
 
 If $\pi_j > \pi_{j+1}$, we need to show that $ y_{ij}- \pi_j \geq y_{i+1,j+1} - \pi_{j+1} $. We consider a lower bound for $y_{ij}$ and an upper bound for $y_{i+1,j+1}$.
 
 For $y_{ij}$, we know that it is the number of $\{1,\dots , i\}$ in the $j-$th row. So, at least, $y_{ij}$ is the number of $j$'s in the $j-$th row and, applying ($\star$), we obtain the following bound:
 \begin{eqnarray}\label{Bound1} 
  y_{ij} \geq \mu_j + (n+1)\cdot \pi_j - c_j .
 \end{eqnarray}
 For $y_{i+1,j+1}$, we will use the general upper bound saying that, at most, we will have as many numbers as boxes in the $(j+1)-$th row of $\mu + (n+1)\cdot \pi$:
 \begin{eqnarray}\label{Bound2}
  y_{i+1,j+1} \leq \mu_{j+1} + (n+1)\cdot \pi_{j+1}.
 \end{eqnarray}
 
 Putting together both bounds, ($\ref{Bound1}$) and ($\ref{Bound2}$), we obtain that
 \begin{eqnarray*}
x_{ij} \geq x_{i+1,j+1} \hspace{0.3cm} \text{ as soon as }  \hspace{0.3cm} n \geq \frac{\mu_{j+1} - \mu_j +c_j}{\pi_j - \pi_{j+1}} .
 \end{eqnarray*}

 \item $x_{ij} \geq 0$ \\
 Due to the other inequalities, it is enough to check it for the elements $x_{ii}=y_{ii}-\pi_i$. In the $i-$th row, there is no numbers from $\{1,.., i-1\}$, so $y_{ii}$ is exactly the number of $i$'s in the $i-$th row and 
 this means that $y_{ii} \geq \mu_i + (n+1)\cdot \pi_i - c_i$. So, we have that
 \begin{eqnarray*}
  x_{ii}  \geq 0 \hspace{0.3cm} \text{ as soon as }  \hspace{0.3cm} n \geq \frac{c_i - \mu_i}{\pi_i} .
 \end{eqnarray*}

\end{itemize}

\end{proof}

\begin{observation}
 As before, the stability property corresponding to the plethysm coefficients $a_{\lambda, \mu +n\cdot \pi}^{\nu + n\cdot |\lambda | \cdot \pi}$ is a consequence of Theorem \ref{R2}. 
\end{observation}

\subsection{Bounds}

In this section we compare the bounds obtained in the previous arguments with previously known bounds derived using other methods.  

We summarize the bounds obtained by Thibon and Carr\'e and by Brion in \cite{CaTh92} and \cite{Br93}, respectively. 
\begin{prop}\label{OtherBounds}
\hfill
\begin{enumerate}
 \item[$(P1)$] In \cite[Theorem 4.2]{CaTh92}, it is proved that the sequence of plethysm coefficients
 \begin{eqnarray*}
  a_{\lambda +(n), \mu}^{\nu + (|\mu| \cdot n)} = \left\langle s_{\lambda +(n)}[s_{\mu}], s_{\nu + (|\mu|\cdot n)} \right\rangle
 \end{eqnarray*}
is identically zero for 
 \begin{eqnarray*}
 n \geq \frac{|\overline{\nu}| \cdot |\mu| }{|\overline{\mu}|} - \nu_1 .
 \end{eqnarray*}
 when $\ell(\mu) > 1$.
 On the other case, if $\ell(\mu)=1$, then the sequence is constant for 
 \begin{eqnarray*}
 n \geq \max \left\{|\overline{\lambda}| + \lambda_2 -2- \left\lfloor \frac{\nu_1}{|\mu|}\right\rfloor ,\ \lambda_2 - 1 + |\overline{\nu}| - \left\lfloor \frac{|\overline{\lambda}| + \nu_1}{|\mu|}\right\rfloor \right\} .
 \end{eqnarray*}
 
 \item[$(Q1)$] In \cite[Theorem, Section 3.1]{Br93}, it is proved that the sequence of general term
  \begin{eqnarray*}
  a_{\lambda +(n), \mu}^{\nu + n \cdot \mu} = \left\langle s_{\lambda + (n)}[s_{\mu }], s_{\nu + n \cdot \mu} \right\rangle
 \end{eqnarray*}
is constant when 
\begin{eqnarray*}
n \geq \lambda_2 - |\lambda| + \left| |\lambda| \cdot \mu - \nu\right| .
\end{eqnarray*}

 \item[$(R1)$] In \cite[Theorem 4.1]{CaTh92}, it is proved that the sequence with general term
 \begin{eqnarray*}
  a_{\lambda, \mu + (n)}^{\nu + (|\lambda| \cdot n)} = \left\langle s_{\lambda}[s_{\mu + (n)}], s_{\nu + (|\lambda|\cdot n)} \right\rangle
 \end{eqnarray*}
 is constant for 
  \begin{eqnarray*}
  n \geq  |\lambda | \cdot \mu_1 + \mu_2 - \nu_1 -1 - \left\lfloor \frac{\nu_1}{|\lambda |} \right\rfloor .
  \end{eqnarray*}
 
 \item[$(R2)$]In \cite[Corollary 1, Section 2.6]{Br93}, it is proved that the sequence of general term
  \begin{eqnarray*}
  a_{\lambda, \mu + n\cdot \pi}^{\nu + n \cdot |\lambda| \cdot \pi} = \left\langle s_{\lambda}[s_{\mu + n\cdot \pi}], s_{\nu + n \cdot |\lambda| \cdot \pi} \right\rangle
 \end{eqnarray*}
is constant when, for all $j$,
\begin{eqnarray*}
n \geq \frac{\mu_{j+1} - \mu_j + c_j}{\pi_j - \pi_{j+1}}
\end{eqnarray*}
where $c_j = |\lambda| \cdot (\mu_1 + \dots + \mu_j ) - (\nu_1 + \dots + \nu_j)$ .
\end{enumerate}
\end{prop}

\begin{observation}
The bounds obtained in the results by Brion, $(Q1), (R2)$, can be found directly in \cite{Br93}. On the other hand, the bounds obtained in the results of Thibon and Carr\'e papers can be found in \cite{CaTh92}, but they need some adaptations due to the notation they used. 
\end{observation}

In our case, we have obtained the following bounds for the $h-$plethysm coefficients, $b_{\lambda \mu}^\nu$, in  the Theorems \ref{Q1}, \ref{P1}, \ref{R1} and \ref{R2}.

\begin{cor}\label{CorBounds}
\hfill
\begin{enumerate}
 \item[(P1)] Consider the sequence $\left( b_{\lambda + (n), \mu}^{\nu + n\cdot \mu}\right)$. Then,
  \begin{itemize}
   \item When $\ell(\mu)=1$, the sequence is constant for $n> |\overline{\nu}|- \lambda_1 -1$.
   \item When $\ell(\mu) >1$, the sequence has limit zero, once $n$ is bigger than $\frac{\mu_1 \cdot |\lambda| - \nu_1}{|\overline{\mu}|}$.
  \end{itemize}
 
 \item[(Q1)]  The sequence $\left( b_{\lambda +(n), \mu}^{\nu + n \cdot |\mu|}\right)$ stabilizes for $n \geq |\lambda| \cdot \|\mu\| - \|\nu\| - \lambda_1$. 
 
 \item[(R1)] The sequence $\left( b_{\lambda, \mu + (n)}^{\nu + (|\lambda|\cdot n)} \right)$ is constant for $n \geq \mu_1 \cdot (|\lambda| - 1) + \mu_2 - \nu_1$.
 
 \item[(R2)] The sequence of general term $b_{\lambda, \mu + n\cdot \pi}^{\nu + n\cdot |\lambda| \cdot \pi}$ is stable for $n \geq \frac{\mu_{j+1} - \mu_j + c_j}{\pi_j - \pi_{j+1}}$, for all $j$, where the $c_j$ are some fixed constants such that \newline
 $c_j > \sum_{i=1}^j \left( |\lambda | \cdot \mu_i - \nu_i \right)$.
\end{enumerate}
\end{cor}

Then, we have the following result for the plethysm coefficients $a_{\lambda(n),\mu(n)}^{\nu(n)}$.
\begin{cor}\label{MyBounds}
 The bounds for the $h-$plethysm coefficients obtained in Corollary \ref{CorBounds} are also the bounds for the plethysm coefficients $a_{\lambda(n),\mu(n)}^{\nu(n)}$ corresponding to the stability properties $(P1)$, $(Q1)$, $(R1)$ and $(R2)$. 
\end{cor}

The proof follows from the following properties:
\begin{enumerate}
\item $\min_{\tau \in \mathfrak{S}_{N^\prime}} \{\omega_1(\tau) \}=0$.

\item $\left|\lambda + \omega(\sigma)\right| = |\lambda| $, for any $\sigma \in \mathfrak{S}_{N} $.

\item $\min_{\tau \in \mathfrak{S}_{N^\prime}} \left\{ \sum_{i=1}^j \omega_i(\tau) \right\} = 0$. 

\item $\min_{\tau \in \mathfrak{S}_{N^\prime}} \left\{ \| \omega(\tau) \| \right\}=0$.
\end{enumerate}

Now we are able to compare both lists of bounds.
\begin{cor}
The following list summarizes the comparison between the bounds obtained in Proposition \ref{OtherBounds} and in Corollary \ref{MyBounds}, after a deep study of the inequalities involved.  
\begin{itemize}
\item[(P1)] If $\ell(\mu)>1$ and if $\mu=(1)$, our bound is better or equal than the bound obtained by Thibon and Carr\'e. In case of $\mu=(m)$, with $m>1$, we have that our bound is better or equal if we are in one of the following cases:
\begin{enumerate}
\item The partition $\nu$ satisfies that $|\overline{\nu}| \leq \frac{|\overline{\lambda}|\cdot (m+1)}{m} -1$ and $|\overline{\nu}| \leq \frac{m\cdot \lambda_2-1}{m-1}$.
\item The partition $\nu$ satisfies that $|\overline{\nu}| > \frac{|\overline{\lambda}|\cdot (m+1)}{m} -1$ and $\nu_1 \leq m\cdot (\lambda_1 + \lambda_2+1) - |\overline{\lambda}|-1$.
\end{enumerate}
\item[(Q1)] The bound obtained by Brion is better or equal than our bound. 

\item[(R1)] Our bound is better than the bound obtained by Thibon and Carr\'e if $\nu_1 \leq |\lambda| \cdot \mu_1-1$. 

\item[(R2)] Both bounds are the same.  
\end{itemize}
\end{cor}

\section*{Final Remarks}

This new approach provides a new proof of stability properties for plethysm coefficients using elementary tools of symmetric functions.  Furthermore, it enhanced the importance of other constants (the $b_{\lambda, \mu}^\nu$ in this paper) that seem interesting by themselves. In particular, it should be possible to evaluate them efficiently, by means of Barvinok's algorithm. Could this lead to more efficient algorithms for computing the plethysm coefficients? Could this approach help us to prove more general stability properties? 

We want to thank Mercedes Rosas, Emmanuel Briand, and Mateusz Micha{\l}ek for some interesting observations and comments. 

\newpage
\bibliographystyle{alpha}
\bibliography{Referencias_Articulos}
\label{sec:biblio}

\end{document}